\documentclass{amsart}
\usepackage{mathrsfs}
\usepackage{mhequ}
\usepackage{graphicx}
\usepackage{tikz}

\newtheorem{thm}{Theorem}

\newtheorem{lem}[thm]{Lemma}
\newtheorem{prop}[thm]{Proposition}
\newtheorem{proposition}[thm]{Proposition}
\theoremstyle{definition}

\theoremstyle{remark}

\newtheorem{remark}[thm]{Remark}
\theoremstyle{theorem}

\theoremstyle{theorem}

\numberwithin{equation}{section}
\numberwithin{thm}{section}

\let\d\partial

\newcommand{\CC}{\mathcal{C}}
\newcommand{\CG}{\mathcal{G}}
\newcommand{\CO}{\mathcal{O}}

\newcommand{\CL}{\mathcal{L}}
\newcommand{\Prob}{\mathbb{P}}
\newcommand{\E}{\mathbb{E}}
\newcommand{\R}{\mathbb{R}}
\newcommand{\eps}{\varepsilon}

\def\eref#1{(\ref{#1})}
\def\scal#1{\langle#1\rangle}
\def\sign{\mathop{sign}}
\def\cD{\mathscr{D}}
\newcommand{\eqdef}{\stackrel{\mbox{\tiny def}}{=}}

\begin{document}

\title[Periodic homogenization with an interface]{Periodic homogenization with an interface: the one-dimensional case}%
\author{Martin Hairer  \and Charles Manson}%
\address{The University of Warwick}%
\email{charliemanson1982@hotmail.co.uk}%
\address{Courant Institute, NYU}%
\email{martin.hairer@courant.nyu.edu}%

\begin{abstract}
We consider a one-dimensional diffusion process with coefficients that are periodic outside of a finite `interface
region'. The question investigated in this article is the limiting long time / large scale
behaviour of such a process under diffusive rescaling. Our main result is that it converges weakly to
a rescaled version of skew Brownian motion, with parameters that can be given explicitly in terms of the coefficients
of the original diffusion.

Our method of proof relies on the framework provided by Freidlin and Wentzell
\cite{MR1245308} for diffusion processes on a graph in order to identify the generator of the limiting
process. The graph in question consists of one vertex representing the interface region and two infinite
segments corresponding to the regions on either side.
\end{abstract}

\maketitle
\thispagestyle{empty}

\section{Introduction}
Consider a diffusion process in $\R^d$ of the type
\begin{equ}[e:defX]
dX(t) = b(X)\,dt + dB(t)\;,\quad X(0) \in \R^d\;,
\end{equ}
where $B$ is a $d$-dimensional Wiener process and $b \colon \R^d \to \R^d$ is periodic and smooth,
and define the diffusively rescaled process $X^\eps(t) = \eps X(t/\eps^2)$.
If $b$ is periodic and satisfies a natural centering condition, then it is well-known that $X^\eps$ converges in law  as $\eps \to 0$
to a Wiener process with a `diffusion tensor' that can be expressed in terms of the solution to a suitable Poisson equation, see for example
the monographs \cite{lions,MR2382139}.

Similar types of homogenization results still hold true if $b$ is not exactly periodic, but of the form $b(X) = \tilde b(X,\eps X)$, for some
smooth function $\tilde b$ that is periodic in its first argument.
In other words, $b$ consists of a slowly varying
component, modulated by fast oscillations. In this case, the limiting process is not a Brownian motion anymore, but can be an
arbitrary diffusion, whose coefficients can again be obtained by a suitable averaging procedure \cite{MR2126985}.
The aim of this article is to consider a somewhat different situation where there is an abrupt change from one type of
periodic behaviour to another, separated by an interface of size order one in the original `microscopic' scale.

To the best of our knowledge, this situation has not been considered before, although a similar problem was studied in \cite{MR2021045}.
In order to keep calculations simple, we restrict ourselves here to the one-dimensional situation. Building on this analysis, we are able to
address the multidimensional case in \cite{MultiD}.
Restricting ourselves to the one-dimensional case considerably simplifies the analysis
due to the following two facts:
\begin{itemize}
\item Any one-dimensional diffusion is reversible, so that its invariant measure can be given explicitly.
\item The `interface' is a zero-dimensional object, so that it cannot exhibit any internal structure in the limit.
\end{itemize}
Before we give a more detailed description of our results, let us try to `guess' what any limiting process $X^0$
should look like, if it exists. Away from the interface, we can apply the existing results on periodic homogenization,
as in \cite{lions,MR2382139}. We can therefore
compute diffusion coefficients $C_\pm$ such that $X^0$ is expected to behave like $C_+ W(t)$ whenever $X^0 > 0$ and like
$C_- W(t)$ whenever $X^0 < 0$, for some Wiener process $W$. One possible way of constructing a Markov process with this property
is to take $X^0(t) = G(W(t))$, where $G\colon \R \to \R$ is given by
\begin{equ}
G(x)  =
\left\{\begin{array}{cl}
	C_+ x & \text{if $x \ge 0$,} \\
	C_- x & \text{otherwise.}
\end{array}\right.
\end{equ}
In turns out that processes of this form do not describe all the possible limiting processes that one can get in the presence of
an interface. The reason why this is so can be seen by comparing the invariant measure of $G(W(t))$ to the invariant measure
of $X^\eps$. Since the invariant measure for $W$ is Lebesgue measure (or multiples thereof), the invariant measure for
$G(W(t))$ is given by
\begin{equ}[e:immu]
\mu(dx) =
\left\{\begin{array}{cl}
	\lambda_+\,dx & \text{if $x > 0$,} \\
	\lambda_-\,dx & \text{otherwise,}
\end{array}\right.\quad \text{with}\quad
\frac{\lambda_+}{\lambda_-} = {C_- \over C_+}\;.
\end{equ}
On the other hand, if we denote the invariant measure for $X^\eps$ by $\mu^\eps = \rho^\eps(x)\,dx$, then
$\rho^\eps$ will typically look as follows:
\begin{equ}
\includegraphics{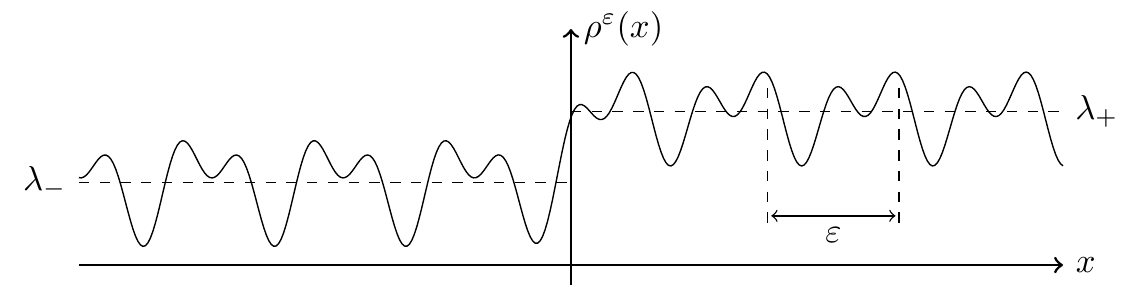}
\end{equ}
It follows that one does indeed have $\mu^\eps \to \mu^0$ as $\eps \to 0$, where $\mu^0$ is of the type \eref{e:immu}, but
the ratio $\lambda_+ / \lambda_-$ depends on the behaviour of $b$, not only away from the interface, but also at the interface.
This can be understood as the process $X^0$ picking up an additional drift, proportional to the local time spent at $0$, that skews
the proportion of time spent on either side of the interface. A Markov process with these properties can be constructed
by applying the function $G$ to a skew-Brownian motion (see for example \cite{MR2280299}) with parameter $p$ for a suitable value of $p$.

An intuitive way of constructing this process goes as follows. First, draw the zeroes of a standard Wiener process on the real
line. These form a Cantor set that partitions the line into countably many disjoint
open intervals. Order them by decreasing length and denote by $I_n$ the length of the $n$th interval.
For each $n \ge 0$, toss an independent biased coin and draw an independent Brownian excursion.
If the coin comes up heads (with probability $p$), fill the interval with the Brownian excursion, scaled horizontally by $I_n$ and
vertically by $C_+\sqrt{I_n}$. Otherwise (with probability $1- p$), fill the interval with the Brownian excursion,
scaled horizontally by $I_n$ and
vertically by $-C_-\sqrt{I_n}$. One can check that the invariant measure for this process is given by \eref{e:immu}, but
with
\begin{equ}[e:relplambda]
{\lambda_+ \over \lambda_-} = {p \,C_- \over (1-p) \,C_+}\;.
\end{equ}
We denote the corresponding process by
$B_{C_\pm,p}(t)$.

This should almost be sufficient to guess the main result of this article. To fix notations, we consider the process $X(t)$
as in \eref{e:defX} and its rescaled version $X^\eps$, and we assume that the drift function $b$ is smooth and
periodic away from an `interface' region $[-\eta,\eta]$. More precisely, we assume that there exist smooth periodic functions $b_{i}:\R\rightarrow\R$, $i \in \{+,-\}$, such that $b_{i}(x+1) = b_{i}(x)$ and such that $b(x) = b_{+}(x-\eta)$ for $x>\eta$
and $b(x) = b_{-}(x+\eta)$ for $x<-\eta$. Additionally, we assume that the functions $b_{i}$ satisfy the centering condition
\begin{equs}
 \int_{0}^{1}b_{i}(x)\,dx = 0\;.
\end{equs}
We also set $V(x) = \int_0^x b(x)\,dx$ for $x \in \R$, so that $\exp(2V(x))\,dx$ is invariant for $X$, and similarly for $V_i$.
Denote by $C_i$ the effective diffusion coefficients for the periodic homogenization
problems corresponding to $b_i$ (see equation \eref{e:defCi} below or \cite{MR2382139} for a more explicit expression).
Define furthermore $\lambda_\pm$ by
\begin{equ}[e:deflambdapm]
\lambda_+ = \int_\eta^{\eta+1} \exp(2V(x))\,dx\;,\qquad
\lambda_- = \int_{-\eta-1}^{-\eta} \exp(2V(x))\,dx\;,
\end{equ}
and let $p \in (0,1)$ be the unique solution to \eref{e:relplambda}. With all these notations
at hand, we have:

\begin{thm}
For any $t>0$, the law of $X^\eps$ converges weakly to the law of $B_{C_\pm,p}$ in the space $\CC([0,t],\R)$.
\end{thm}

\begin{remark}
In order to keep notations simple, we have assumed that the diffusion coefficient of $X$ is constant and equal to $1$.
The case of a non-constant, but smooth and uniformly elliptic diffusion coefficient can be treated in exactly the same way, noting that
it reduces to the case treated here after a time change that can easily be controlled.
\end{remark}

\begin{remark}
A natural extension of the results presented in this article is case of a random potential, along the lines of the situation first considered in 
\cite{MR659505}. While the heuristic argument presented in the introduction still applies, the method of proof considered here does not seem to apply readily.

An alternative method would be to consider an injective harmonic function $H_\eps$ for $X^\eps$, so that $H_\eps(X^\eps(t))$ is
a martingale for every $\eps$. In the situation at hand, one can take
\begin{equ}
H_\eps(x) = \int_0^x e^{2V(y/\eps)}\,dy\;.
\end{equ}
One might hope that in this case it is possible to show that the process $H^\eps(t) = H_\eps(X^\eps(t))$ then converges as $\eps \to 0$
to the martingale $H(t)$ given by
\begin{equ}
H(t) = \int_0^t A(H(s))\,dB(s)\;,
\end{equ}
where $A(H)$ takes one constant value for $H>0$ and a different constant value for $H<0$ (see also Section~\ref{sec:unique} below).
This type of approach has been successfully applied to a number of multiscale problems, 
including homogenisation on fractals, see for example \cite{MR1218305,MR1474004,MR1376343,MR1632945,MR2016606,MR2083710}.
\end{remark}

The proof of the weak convergence of the probability distributions on $\CC[0,\infty)$ associated to $X^{\eps}_{x}$ presented in
this article will rely heavily on the 1993 paper by Freidlin and Wentzell \cite{MR1245308}, in which the authors consider a `fast'
Hamiltonian system perturbed by a `slow' diffusion. Theorems 2.1 and 4.1 from \cite{MR1245308} provide a general framework
for proving
first the tightness and then the convergence of a
family of probability distributions on $\CC[0,\infty)$.

We will start by showing tightness of our family of processes in Section~\ref{sec:tight}.
Once tightness is established, we show in Section~\ref{sec:convergence}
that every limiting process
$X^0$ solves the martingale problem associated to a certain generator. Finally, we show in Section~\ref{sec:unique} that this
martingale problem has a unique solution which is precisely the rescaled skew-Brownian motion, thus concluding the proof.

\subsection*{Acknowledgements}

{\small
We would like to thank the anonymous referees for carefully reading a preliminary version of this article and 
encouraging us to shorten some of our arguments considerably. The research of MH was supported by 
an EPSRC Advanced Research Fellowship (grant number EP/D071593/1) and a Royal Society Wolfson Research Merit Award.
}

\section{Proving tightness}
\label{sec:tight}

The main result of this section is the following:
\begin{thm}\label{xtightness}
The family of probability measures on $\CC([0,\infty),\R)$ given by the laws of $X^{\eps}_{x}$ for $\eps \in (0,1]$ is tight.
\end{thm}

\begin{proof}
Given that the initial condition is kept fixed at one single point across the entire family of laws,
tightness follows from uniformity in the modulus of continuity over $\eps \in (0,1]$.

For $x \in \R$ and $\rho > 0$, denote by $\tau_\rho^x$ the first exit time of the canonical process from the interval
$[x-\rho, x+\rho]$. We also denote by $\Prob_{x,\eps}$ the law of $X^\eps_x$.
It then follows immediately from the proof of \cite[Theorem 1.4.6]{MR532498} that a sufficient criterion for tightness is that,
for every $\rho > 0$, there exists a constant $A_\rho$ such that the bound
\begin{equ}[delta]
	\Prob_{x,\eps}(\,\tau_\rho^x \leq \delta) \leq \delta A_{\rho} \;,
\end{equ}
holds uniformly over all $x\in\R$, $\eps \in (0,1]$, and $\delta \in \R$. Before we proceed, we note the following two crucial facts:
\begin{enumerate}
\item \label{1} It follows from the periodic case \cite[Section~3.4]{lions} that 
(\ref{delta}) holds uniformly for $x \notin (-\eps\eta-\rho,\eps\eta+\rho)$. We denote the corresponding constants by $A_\rho^1$.
\item  A standard martingale argument as in \cite[Section~1.4]{MR532498} shows that, for every $\eps_0 > 0$, the bound \eref{delta} 
holds uniformly over all $x\in \R$,
provided that we restrict ourselves to $\eps \in [\eps_0, 1]$. For the sequel of the proof it will be convenient to make the choice
$\eps_0 = \rho/(4\eta)$ and we denote the corresponding constants by $A_\rho^2$.
\end{enumerate}

Combining these two facts, we see that it remains to find a family of constants $A_\rho^3$ such that \eref{delta} holds for
$x \in (-\eps\eta-\rho,\eps\eta+\rho)$ and for $\eps < \rho/(4\eta)$. At this stage we note that since $\rho$ is greater than twice
the width of the interface, for every $x \in \R$ there exist two points $\tilde x_{\pm}$ with the following two properties:
\begin{enumerate}
\item The process started at $x$ has to hit either $\tilde x_+$ or $\tilde x_-$ before
it can reach the boundary of the interval $[x-\rho, x+\rho]$.
\item The intervals $I_\pm = [\tilde x_\pm - {\rho\over 8}, \tilde x_\pm + {\rho\over 8}]$ satisfy 
$I_\pm \cap (-\eta\eps, \eta\eps) = \emptyset$ and $I_\pm \subset [x-\rho, x+\rho]$. 
\end{enumerate}
Restarting the process when it hits one of the $\tilde x_\pm$, it follows from the strong Markov property that 
we can choose $A_\rho^3 = A^{1}_{\rho/8}$, which concludes the proof by setting $A_\rho = \max\{A_\rho^i\}$.
\end{proof}

\section{Convergence of the laws}
 \label{sec:convergence}

As usual in the theory of homogenization, we do not show directly that the processes $X^\eps$ converge to a limit, 
but we first introducing a compensator $g \colon \R \to \R$ that `kills' the strong drift of the rescaled
process and consider instead the family of processes
 \begin{equ}[e:defY]
 Y^{\eps}(t) = X^{\eps}(t) + \eps g \Bigl( \frac{X^{\eps}(t)}{\eps} \Bigr)\;.
 \end{equ}
Since we will choose $g$ to be a bounded function, the weak convergence in the space of continuous functions
of the laws of $X^\eps$ to some limiting process is equivalent to that of the $Y^\eps$.

In order to construct $g$, let $\CL_i$ denote the generator of the diffusion with drift $b_i$, that is $\CL_i = {1\over 2 }\d_x^2 + b_i(x)\d_x$,
and denote by $\mu_i(dx) = Z^{-1}\exp(2V_i(x))\,dx$ the corresponding invariant probability measure on $[0,1]$.
We then denote by $g_i$ the unique smooth function solving
\begin{equ}[e:defgi]
\CL_i g_i = -b_i\;,\quad \int_0^1 g_i(x)\,\mu_i(dx) = 0\;.
\end{equ}
Since $b$ is assumed to be centred on either side of the interface, such a function exists (and is unique) by the 
Fredholm alternative.
We now choose \textit{any} smooth function $g\colon \R \to \R$ such that $g(x) = g_-(x+\eta)$ for $x \in (-\infty, -\eta)$ and 
$g(x) = g_+(x-\eta)$ for $x\in(\eta, \infty)$, with a smooth joining region in between.

The main ingredient in our proof of convergence will be
\cite[Theorem~4.1]{MR1245308}, which is used in conjunction with the previous tightness result to identify
the weak limit points of the family of probability distributions as the solutions to a martingale problem.
The aim of this section is to explain how to fit our problem into the framework of \cite{MR1245308} and to verify the
assumptions of their main convergence theorem.

Before we proceed, let us recall what is understood by the ``martingale problem'' corresponding to some operator $A$
(see for example \cite{MR838085}), and let us try to guess what the generator $A$ for the limiting process is expected to be.
Let $X$ be a Polish (i.e.\ complete separable metric) space; $\CC[\,0, \infty)$, the space of all continuous functions on $[0, \infty)$
with values in X. For any subset $I \subset [0,\infty)$, denote by $\mathscr{F}_{I}$ the $\sigma$-algebra of
subsets of $\CC[\,0, \infty)$ generated by the sets $\{x \in \CC[\,0, \infty): x(s) \in B\}$, where $s \in I$ and $B\subset X$ is
an arbitrary Borel set. We also denote by $\CC(X)$
    the space of all continuous real-valued functions on $X$.

   Let $A$ be a linear operator on $\CC(X)$, defined on a subspace $\cD(A) \subseteq \CC(X)$. We will say that a probability measure $\mathbb{P}$, on $\bigl(\CC[\,0, \infty), \mathscr{F}_{[\,0,\, \infty)}\bigr)$, is a solution to the martingale problem corresponding to $A$, starting from a point $x_{0} \in X$, if
\begin{equ}
	\mathbb{P}\{x\,:\, x(0) = x_{0}\} = 1 \label{3.2}
\end{equ}
and, for any $f \in \cD(A)$, the random function defined on the probability space $\bigl(\CC[\,0, \infty), \mathscr{F}_{[\,0,\, \infty)}, \mathbb{P}\bigr)$ by
\begin{equ}
    f\bigl(x(t)\bigr)  - \int_{0}^{t}\bigl(Af\bigr)\bigl(x(s)\bigr)\,ds, \text{ $t \in [\,0, \infty)$}\;, \label{3.1}
\end{equ}
is a martingale with respect to the filtration $\bigl\{\mathscr{F}_{[0,t]}\bigr\}_{t > 0}$.

What do we expect the operator $A$ to be given by in our case? On either side of the interface, we argued in the introduction that the limiting
process should be given by Brownian motion, scaled by factors $C_\pm$ respectively. Therefore, one would expect $A$ to be given
by
\begin{equ}[e:defA]
\bigl(Af\bigr)(x) =
\left\{\begin{array}{cl}
	{1\over 2} C_-^2 \d_x^2 f(x) & \text{if $x < 0$,} \\[0.4em]
	{1\over 2} C_+^2 \d_x^2 f(x) & \text{otherwise,}
\end{array}\right.
\end{equ}
and the domain $\cD(A)$ to contain functions that are $\CC^2$ away from the origin. This however does not take into account
for the ``skewing'', which should be encoded in the behaviour of functions in $\cD(A)$ at the origin.

Since the limiting process spends zero time at the origin (the invariant measure is continuous with respect to Lebesgue measure),
it was shown in \cite{MR2280299} that the possible behaviours at the origin are given by matching conditions for the first derivatives of functions
belonging to $\cD(A)$. We know from the introduction that the invariant measure of the limiting process is proportional to Lebesgue measure
on either side of the origin, with proportionality constants $\lambda_\pm$. We should therefore have the identity
\begin{equ}
\lambda_- \int_{-\infty}^{0}Af(x)\,dx + \lambda_+ \int_{0}^{\infty}Af(x)\,dx = 0\;,
 \end{equ}
for every function $f \in \cD(A)$. Using \eref{e:defA}, we thus obtain
 \begin{equ}
\lambda_-  C_{-}^{2} \int_{-\infty}^{0}f''(x)\,dx + \lambda_+ C_+^2 \int_{0}^{\infty}f''(x)\,dx = 0\;.
 \end{equ}
Integrating by parts, this yields (for say compactly supported test functions $f$) the condition
 \begin{equ}[e:dcon]
\lambda_- C_{-}^{2}f'(0^{-}) = \lambda_+ C_{+}^{2}f'(0^{+})\;.
 \end{equ}
This is exactly the general form of a generator produced by Theorem 4.1 in Freidlin and Wentzell \cite{MR1245308}, a differential operator on the regions away from some distinguished points termed nodes, combined with a restriction on the ratios of the limits of
the derivatives at this point.

The main theorem of this section that is also very closely related to the main theorem of the article is as follows:
\begin{thm}\label{Main Theorem of Section 3}
Let $C_\pm$ be given by
 \begin{equ}[e:defCi]
 C_{\pm}^2=\int_{0}^{1}(1+g_{\pm}(x))^{2}\,\mu_{\pm}(dx)\;,
 \end{equ}
 where $g_\pm$ and $\mu_\pm$ are as in \eref{e:defgi}. Let $A$ be given by \eref{e:defA} and let $\cD(A)$ be the set of continuous functions $f \colon \R \to \R$, vanishing at infinity, that are $\CC^2$ away from $0$ and that satisfy the condition \eref{e:dcon} at the
 origin.

Then, every limit point of the family of processes $Y^{\eps}_{x}$ is solution to the martingale problem corresponding to $A$.
\end{thm}

As already mentioned, our main ingredient is \cite[Theorem 4.1]{MR1245308} applied to the sequence of processes $Y^{\eps}$
as defined in \eref{e:defY}. For completeness, we give a simplified statement of this result adapted to the situation at hand:

 \begin{thm}[Freidlin \& Wentzell] \label{4111}
 Let $\CL_{i}$, $i = \pm$, be elliptic second order differential operators with smooth coefficients on
 $I_{i}$, $I_{+}=[0,\infty)$, $I_{-}=(-\infty,0]$, and let $Y^{\eps}$
 be a family of real-valued processes satisfying the strong Markov property. For some fixed $\tilde\eta>0$, let  $\tau^{\eps}$
 be the first hitting time of the set $(-\eps\tilde\eta,\eps\tilde\eta)$ by $Y^{\eps}$.

 Assume that there exists a function $k \colon\R_+ \to \R_+$ with $\lim_{\eps \rightarrow 0}k(\eps) = 0$ such that,
  for any function $f \in \CC^\infty_v(I_i)$ and for any $\lambda > 0$, one has the bound\footnote{We denote by $\CC_v^\infty$ the space
  of smooth functions that vanish at infinity, together with all of their derivatives.}
 \begin{equ}
 \mathbb{E}_{y}\biggl[\,e^{-\lambda \tau^{\eps}}f\bigl(Y^{\eps}(\tau^{\eps})\bigr) - f\bigl(y)+ \int_{0}^{\tau^{\eps}}e^{-\lambda t}\Bigl(\lambda f\bigl(Y^{\eps}(t)\bigr) - \CL_{i}f\bigl(Y^{\eps}(t)\bigr)\Bigr)\,dt\,\biggr]
 = \CO\bigl(k(\eps)\bigr) \label{2.1}
 \end{equ}
 as $\eps \rightarrow 0$, uniformly with respect to $y \in I_{i}$.
Assume furthermore that there exists a function $\delta \colon\R_+ \to \R_+$ with  $\lim_{\eps \rightarrow 0}\delta(\eps) = 0$
and $\lim_{\eps \rightarrow 0} \delta(\eps)/k(\eps) \rightarrow \infty$ such that, for any $\lambda > 0$,
  \begin{equ}
  \mathbb{E}_{y} \biggl[\int_{0}^{\infty}e^{-\lambda t} 1_{(-\delta\,,\,\delta)}\bigl(Y^{\eps}(t)\bigr)\,dt \,\biggr]\rightarrow 0 \label{2.2}
   \end{equ}
  as $\eps \rightarrow 0$, uniformly over all $y \in R$. Finally, writing $\sigma^{\delta}$ for the first hitting time of the
  set $(-\infty,-\delta)\cup(\delta,\infty)$ by $Y^\eps$, assume that there exist $p_i\ge 0$ with $p_{-} + p_+ = 1$ such that
  \begin{equ}
  	\mathbb{P}_{y}\bigl[Y^{\eps}(\sigma^{\delta}) \in I_{i}\bigr] \rightarrow p_{i}\;,\quad i\in \{+,-\} \;, \label{2.3}
  \end{equ}
  uniformly for $y \in (-\eps\tilde\eta,\eps\tilde\eta)$.

Let now $A$ be the operator defined by $Af(x) =  \CL_{i}f(x)$ for $x \in I_i$ with domain $\cD(A)$ consisting of functions $f$ such that $f|_{I_i} \in \CC^\infty_v(I_i)$ and such that the `matching condition' $p_+ f'(0^+) = p_- f'(0^-)$ holds.
Then for any fixed $t_{0} \geq 0$, any $\lambda > 0$, and any $f \in \cD(A)$, the bound
  \begin{equ}[e:resultconv]
  \text{\rm ess sup} \,\biggl|\, \int_{t_{0}}^{\infty}e^{-\lambda t}\mathbb{E}_{y}\Bigl[\,\lambda f \bigl(Y^{\eps}(t)\bigl) - Af\bigl(Y^{\eps}(t)\bigr)\Bigr|\, \mathscr{F}_{[\,0,\,t_{0}]} \Bigr]\,dt - e^{-\lambda t_{0}}f\bigl(Y^{\eps}(t_{0})\bigr)\biggr| \rightarrow 0 \label{2.4}
  \end{equ}
holds  as $\eps \rightarrow 0$, uniformly for all $y\in \R$.
\end{thm}

\begin{remark}
The version of Theorem~\ref{4111} stated in \cite{MR1245308} does actually treat more general diffusions on graphs, but assumes
that the edges of the graph are finite. This is not really a restriction, since $I_+$ is in bijection with $[0,1)$ (and similarly for $I_-$) and we can
simply add non-reachable vertices at $\pm 1$ to turn our process into a process on a finite graph.
\end{remark}

\begin{remark}
As can be seen by combining \eref{e:relplambda} and \eref{e:dcon}, the probabilities $p_\pm$ appearing in the statement
of Theorem~\ref{4111} are not quite the same in general as the probabilities $\tilde p_\pm = \{p,1-p\}$ appearing in the construction of skew Brownian
motion in the introduction. The relation between them is given by ${p C_+ \over (1-p) C_-} = {p_+ \over p_-}$.
The reason is that $p_\pm$ give the respective probabilities of hitting two points located at a \textit{fixed} distance from the `interface', whereas the
non-trivial scaling of the Brownian bridges on either side of the interface means that $\tilde p_\pm$ give the probabilities of hitting
two points whose distances from the interface have the ratio $C_+ / C_-$.
\end{remark}


Most of the remainder of this section is devoted to the fact that:

\begin{proposition}\label{prop:assum}
The family of processes $Y^\eps$ given by \eref{e:defY} satisfies the assumptions of Theorem~\ref{4111} with
$\CL_\pm = {1\over 2} C_\pm \d_x^2$ and $p_\pm$ defined by the relations
\begin{equ}
p_+ + p_- = 1\;,\qquad {p_+\over p_-} = {\lambda_+ C_+^2 \over \lambda_- C_-^2}\;,
\end{equ}
and $\lambda_\pm$ as in \eref{e:deflambdapm}.
\end{proposition}

This yields the

\begin{proof}[Proof of Theorem \ref{Main Theorem of Section 3}]
Before we start, let us remark that the initial condition $y$ for the corrected process
$Y^\eps$ and the initial condition $x$ for the original process $X^\eps$ are related by $y = x + \eps g(x/\eps)$.

Note also that, thanks to the identity $\int_{t_0}^\infty e^{-\lambda s} F(s)\,ds = \int_{t_0}^\infty \lambda e^{-\lambda t} \int_{t_0}^t F(s)\,ds\,dt$
valid for any bounded measurable function $F$, the left hand side in \eref{2.4} can be written as
\begin{equs}
\Delta(\eps) &= \int_{t_0}^\infty \lambda e^{-\lambda t} \E_y \Bigl(f(Y^\eps(t)) - f(Y^\eps(t_0)) - \int_{t_0}^t Af(Y^\eps(s))\,ds\,\Big|\,
\mathscr{F}_{[\,0,\,t_{0}]} \Bigr)\,dt\\
&\eqdef  \int_{t_0}^\infty \lambda e^{-\lambda t} \E_y \bigl(\CG_f(Y^\eps,t_0,t) \big|\,
\mathscr{F}_{[\,0,\,t_{0}]} \bigr)\,dt\;.
\end{equs}

We have already established the weak precompactness of the family $\{\mathbb{P}_{x,\eps},\, \eps > 0\}$ in the space $\CC(\R_+,\R)$.
The uniformity in $x$ of the convergence of $\Delta(\eps)$ to $0$ then implies that
 for any $n$, any $0 \leq t_{1} < \cdots < t_{n} \leq t_{0}$, and any bounded measurable function $G(x_{1}, \ldots, x_{n})$, $x_{i} \in \R$,
   \begin{equ}
   \Bigl|\, \mathbb{E}_y\Bigl(\,G(Y^\eps(t_{1}), \ldots, Y^\eps(t_{n})) \cdot \int_{t_{0}}^{\infty}\lambda e^{- \lambda t}\CG_f(Y^\eps,t_0,t)\,dt\Bigr) \Bigr|
    \leq\, \sup\bigl|G \bigr| \cdot \Delta (\eps)\;. \label{3.4}
    \end{equ}
   If we furthermore assume that $G$ is continuous, then the expression inside the expectation is a continuous function on $\CC(\R_+,\R)$,
   so that any accumulation point $X^0$ satisfies
   \begin{equs} \int_{t_{0}}^{\infty} \lambda e^{- \lambda t}\E\Bigl( G\bigl(X^0(t_{1}), \ldots,X^0(t_{n})\bigr) \CG_f(X^0,t_0,t)\Bigr)\,dt = 0\;. \label{3.5} \end{equs}
   Since the integrand is a continuous function of $t$ and a continuous function is determined uniquely by its Laplace transform,
 this implies that $\mathbb{E}\bigl(G(\ldots) \CG_f(X^0,t_0,t) \bigr) = 0$ for all $n$ and $0 \leq t_{1} < \cdots <t_{n} \leq t_{0}$, so that
in particular the random function $f(X^0(t)) - \int_0^t Af(X^0(s))\,ds$ is indeed a martingale in the filtration generated by
the process $X^0$.

Since the laws of the starting points of $X^\eps$ are all equal to $\delta_x$ by construction,
we conclude that the law of $X^0$ is indeed a solution of the martingale problem corresponding to $A$, starting from $x_{0}$.
\end{proof}

\begin{proof}[Proof of Proposition~\ref{prop:assum}]
The proofs of \eref{2.1}, \eref{2.2} and \eref{2.3} will be given as three separate propositions.

\begin{prop}\label{prop:firstbound}
There exists $\tilde \eta > 0$ such that the process $Y^{\eps}(t)$ satisfies \eref{2.1} with $k(\eps) = {\eps}$, that is,
\begin{equs}\mathbb{E}_{y}\Bigl[\,e^{-\lambda \tau^{\eps}}f\bigl(Y^{\eps}(\tau^{\eps})\bigr) - f(x) + \int_{0}^{\tau^{\eps}}e^{-\lambda t}\Bigl(\lambda f\bigl(Y^{\eps}(t)\bigr) - {1\over 2}C_+^2 f''\bigl(Y^{\eps}(t)\bigr)\Bigr)\,dt\Bigr] = \CO\bigl(\eps\bigr)\;,
\end{equs}
for every function $f \in \CC^\infty_0(\R_+)$, uniformly in $y \in [\tilde\eta \eps, \infty)$, and similarly for the left side of the interface.
\end{prop}

\begin{proof}
The treatment of both sides of the interface is identical, so we restrict ourselves to $\R_+$. As before, the initial condition $y$ for the corrected process
$Y^\eps$ and the initial condition $x$ for the original process $X^\eps$ are related by $y = x + \eps g(x/\eps)$.
Note that one has $g_\pm'(x) \neq -1$ for any $x$ since otherwise, by uniqueness of the solutions to the ODE $g'' = -2(1+g')b$, this would entail
that $g_\pm'(x) = -1$ over the whole interval $[0,1]$, in contradiction with the periodic boundary conditions.

By possibly making $\eta$ slightly larger, we can (and will from now on) therefore assume that $g'(x) > -1$
uniformly over $x \in \R$, so that the correspondence $x \leftrightarrow y$ is a
bijection. Since $g$ is bounded, this
shows that one can find $\tilde \eta > 0$ so that $y \not\in [-\eps \tilde\eta, \eps \tilde \eta]$ implies that $x \not\in [-\eps \eta, \eps \eta]$.
In particular, fixing such a value for $\tilde \eta$ from now on, we see that the drift vanishes in the SDE satisfied by $Y^\eps$, provided
that we consider the process only up to time $\tau^\eps$.

Using the integration by parts formula and It\^{o}'s formula for each $Y_{y}^{\eps}$ we get
    \begin{equs}e^{-\lambda t}f\bigl(Y_{y}^{\eps}(\tau^{\eps})\bigr) &= f(x) + \int_{0}^{\tau^{\eps}}e^{-\lambda s}\bigl(1+g_+'(\eps^{-1}X_{x}^\eps(s))\bigr) f'\bigl(Y_{y}^{\eps}(s) \bigr)\,dB_{s} \label{102}\\
  &\quad  - \int_{0}^{\tau^{\eps}} e^{-\lambda s}\Bigl[\,\lambda f\bigl(Y_{y}^{\eps}(s)\bigr) + \frac{1}{2}\bigl(1+ g_+'(\eps^{-1}X_{x}^\eps(s))\bigr)^{2} f''\bigl(Y_{y}^{\eps}(s)\bigr)\Bigr]\,ds
    \end{equs}

Since the expectation of the stochastic integral vanishes (both $f'$ and $g_+'$ are uniformly bounded), all that remains to be
shown is that the last term in the above equation
converges at rate $\eps$ to the same term with $(1+g')^2$ replaced by $C_+^2$.

This will be a consequence of the following result (variants of which are quite standard in the theory of periodic
homogenization), which considers the fully periodic case. It is sufficient to consider this case in the situation at hand since
we restrict ourselves to times before $\tau^\eps$, so that the process does not `see' the interface.
  \begin{lem}\label{lem:avzero}
Let $b \colon \R^n \to \R^n$ be smooth and periodic with fundamental domain $\Lambda \subset \R^n$, and denote by $\mu$ the (unique) probability
measure on $\Lambda$ invariant for the SDE
\begin{equ}[e:SDE]
dX(t) = b(X(t))\,dt + dB_t\;,\quad X(0) = x\;,
\end{equ}
where $B$ is a standard $d$-dimensional Wiener process.
Assume furthermore that $\int_\Lambda b(x)\,\mu(dx) = 0$. (This condition will be referred to in the sequel as $b$ being centred.)

Let $h \colon \R^d \to \R$ be any smooth function that is periodic with fundamental domain $\Lambda$ and centred. Let furthermore $X^\eps(t) = \eps X(t/\eps^2)$
and let $\tau_\eps$ be a family of (possibly infinite) stopping times with respect to the natural filtration of $B$. Then, for every $F \in \CC_0^4(\R^d,\R)$
there exists $C>0$ independent of $\tau_\eps$ such that the bound
\begin{equs}
	\E_{x}\Biggl[\int_{0}^{\tau_\eps} e^{-\lambda s}F\bigl(X^{\eps}(s)\bigr)h\biggl(\frac{X^{\eps}(s)}{\eps}\biggr)\,ds\Biggr] \le C\eps\;,
\end{equs}
holds for any $\eps \in (0,1]$, uniformly in $x$.
\end{lem}
\begin{proof}
Denote by $\CL = {1\over 2}\Delta + \scal{b(x),\nabla}$ the generator of \eref{e:SDE} and let $g$ be the unique periodic centred solution
to $\CL g = h$. (Such a solution exists by the Fredholm alternative.)
Applying It\^{o}'s formula to the process
 \begin{equs}
    \eps^{2} e^{-\lambda t}F\bigl(X^{\eps}(t)\bigr)g\biggl(\frac{X^{\eps}(t)}{\eps}\biggr)\;,
 \end{equs}
we obtain the identity
\begin{equs}
    \eps^{2}e^{-\lambda \tau_\eps}F\bigl(X^\eps(\tau_\eps)\bigr)&g\biggl(\frac{X^\eps(\tau_\eps)}{\eps}\biggr) =
     \int_{0}^{\tau_\eps}e^{-\lambda s}F\bigl(X^{\eps}(s)\bigr)h\biggl(\frac{X^{\eps}(s)}{\eps}\biggr)\,ds \\
    &\qquad + \eps^{2}\int_{0}^{\tau_\eps}-\lambda e^{-\lambda s} F\bigl(X(s)\bigr)g\biggl(\frac{X(s)}{\eps}\biggr) \,ds \\
    &\qquad +\eps\int_{0}^{\tau_\eps}e^{-\lambda s}b\biggl(\frac{X^{\eps}(s)}{\eps}\biggr)\ldotp\nabla F\bigl(X^{\eps}(s)\bigr)g\biggl(\frac{X^{\eps}(s)}{\eps}\biggr)\,ds \\
    &\qquad + \frac{1}{2}\,\eps^{2}\int_{0}^{\tau_\eps}e^{-\lambda s}\Delta F\bigl(X^{\eps}(s)\bigr)g\biggl(\frac{X^{\eps}(s)}{\eps}\biggr)\,ds \\
    &\qquad +\eps^{2}\int_{0}^{\tau_\eps}e^{-\lambda s}\nabla F\bigl(X^{\eps}(s)\bigr)g\biggl(\frac{X^{\eps}(s)}{\eps}\biggr)\,dB_{s} \\
    &\qquad +\eps \int_{0}^{\tau_\eps}e^{-\lambda s}F\bigl(X^{\eps}(s)\bigr)\nabla g\biggl(\frac{X^{\eps}(s)}{\eps}\biggr)\,dB_{s} \\
    &\qquad +\eps\int_{0}^{\tau_\eps} e^{-\lambda s}\biggl(\nabla F\bigl(X^{\eps}(s)\bigr)\ldotp\nabla g\Bigl(\frac{X^{\eps}(s)}{\eps}\Bigr)\biggr)\,ds\;.
\end{equs}
The claim then follows by taking expectations and noting that all the functions of $X^\eps$ appearing in the various terms
are uniformly bounded.
\end{proof}
Returning to the proof of Proposition~\ref{prop:firstbound}, we first note that $f''(Y_y^\eps) = f''(X_x^\eps) + \CO(\eps)$, so that we
can replace  $f''(Y_y^\eps)$ by $f''(X_x^\eps)$ in \eref{102}, up to errors of $\CO(\eps)$.
Applying Lemma~\ref{lem:avzero} with $F = f''$ and $h = (1+ g_+')^2 - C_+^2$, the claim then follows at once.
\end{proof}

\begin{prop}\label{prop:timeint}
The convergence
    \begin{equs}\mathbb{E}\biggl[\int_{0}^{\infty}e^{-\lambda t}1_{(-\sqrt \eps,\sqrt \eps)}\bigl(Y_y^{\eps}(t)\bigr)\,dt\biggr] \rightarrow 0 \label{4.2forz}\end{equs}
takes place as $\eps \rightarrow 0$, uniformly in the initial point $y \in \R$.
 In particular,   (\ref{2.2}) holds with $\delta(\eps) = \sqrt \eps$.
\end{prop}

\begin{proof}
The main idea is to first perform a time-change that turns the diffusion coefficient of $Y^\eps$ into $1$ and to then compare the
resulting process to the process $V^\eps$ which is the solution to
\begin{equ}[e:defV]
dV^\eps = b^{\eps}_{V}(V^\eps)\,dt + dB_t\;,
\end{equ}
where the drift $b^{\eps}_{V}$ is given by
\begin{equ}[e:defbV]
b^{\eps}_{V}(x)\, =
 \begin{cases}
 -\frac{C_V}{\eps}& \quad\text{ for $0 \leq x \leq \hat \eta \eps$,}
\\
 \frac{C_V}{\eps}& \quad\text{ for $-\hat \eta \eps \leq x < 0$,}\\
 0 &  \quad\text{ otherwise,}
 \end{cases} \end{equ}
for $C_V$ and $\hat \eta$ some positive constants independent of $\eps$ to be determined below.
An explicit resolvent equation then allows one to show that
\eref{4.2forz} with $Y$ replaced by $V$ tends to zero as $\eps \rightarrow 0$, uniformly in the initial point.

First, let us start with the time change. As in the proof of Proposition~\ref{prop:firstbound}, we assume that $g$ is chosen in such a way that
$g'$ is bounded away (from below) from $-1$, so that there exists a constant $\phi>0$ such that $g'(x) \ge \phi - 1$ for every $x \in \R$.
In order to turn the diffusion coefficient of $Y^\eps$ into $1$, we use the time change associated with the quadratic variation of the $Y^{\eps}$,
\begin{equs}
\langle Y^{\eps}, Y^{\eps} \rangle(t) = \int_{0}^{t} \Bigl(1+g'\Bigl(\frac{X^{\eps}(s)}{\eps} \Bigr) \Bigr)^{2}\,ds\;,
\end{equs}
thus setting
\begin{equs}
C^{\eps}_{t} = \inf\bigl\{t'>0: \langle Y^{\eps}, Y^{\eps} \rangle(t') > t\bigr\}\;.
 \end{equs}
Defining the function $\hat b = b + \CL g$, the process $Z^\eps(t) = Y^{\eps}(C^{\eps}_{t})$ then satisfies the equation
\begin{equ}
	Z^{\eps}(t) = y + \int_{0}^{C^{\eps}_{t}} \frac{1}{\eps}\,\hat b \Bigl(\frac{X^{\eps}(s)}{\eps} \Bigr)\,ds + \int_{0}^{C^{\eps}_{t}}\Bigl(1+g'\Bigl(\frac{X^{\eps}(s)}{\eps} \Bigr)\Bigr)\,dB_{s} \;.
\end{equ}
Note that the time-change was defined precisely in such a way that the second term in this expression is equal to some Brownian
motion $B^\eps(t)$. Inserting the expression for the time change, the first term can be rewritten as
    \begin{equ}
    \int_{0}^{C^{\eps}_{t}}\frac{1}{\eps}\,\hat b \Bigl(\frac{X^{\eps}(s)}{\eps} \Bigr) \,ds =
     \int_{0}^{t}\frac{1}{\eps}\,\hat b\Bigl(\frac{X^{\eps}(C^{\eps}_{s})}{\eps}\Bigr) \Bigl(1+g'\Bigl(\frac{X^{\eps}(C^\eps_s)}{\eps} \Bigr)\Bigr)^{-2}\,ds \;.
     \end{equ}
It follows that the drift term is non-zero only when the time-changed process occupies the region $(-\eps\eta-\eps\|g\|_{\infty},\,\eps\eta+\eps\|g\|_{\infty})$ just as for the non time-changed process. To summarise, there exists a drift $\tilde b$ bounded uniformly by
$C_V\over \eps$ for some constant $C_V > 0$ and vanishing outside of $(-\tilde \eta \eps, \tilde \eta \eps)$ for $\tilde \eta = \eta + \|g\|_\infty$,
as well as a Brownian motion $B^\eps$, so that the process $Z^\eps_x$ satisfies the SDE
\begin{equ}[e:equZ]
dZ^\eps_x = \tilde b(Z_x^\eps)\,dt + dB^\eps_t\;,\qquad Z^\eps_x(0) = y \;.
\end{equ}

Now, look at how the time change affects the expression (\ref{4.2forz}), where we set $G^\eps = (-\sqrt \eps, \sqrt \eps)$:
    \begin{equs}
    \int_{0}^{\infty}e^{-\lambda t}1_{G^{\eps}}\bigl(Y^{\eps}(t)\bigr)\,dt &=  \int_{0}^{C^{\eps}_{\infty}}e^{-\lambda t}1_{G^{\eps}}\bigl(Y^{\eps}(t)\bigr)\,dt\\
      &= \int_{0}^{\infty}e^{-\lambda t}1_{G^{\eps}}\bigl(Z^{\eps}(t)\bigr)\Bigl(1+g'\Bigl(\frac{X^{\eps}(C_t^\eps)}{\eps} \Bigr)\Bigr)^{-2}\,dt\\
       &\leq \sup_{x \in \mathbb{R}}\bigl(1+g'(x)\bigr)^{-2}\int_{0}^{\infty}e^{-\lambda t}1_{G^{\eps}}\bigl(Z^{\eps}(t)\bigr)\,dt\;.
        \end{equs}
Hence if it can be shown that,
    \begin{equ}[e:boundWanted]
    \E_y\biggl[\int_{0}^{\infty}e^{-\lambda t}1_{G^{\eps}}\bigl(Z^{\eps}(t)\bigr)\,dt\biggr] \rightarrow 0
    \end{equ}
uniformly in the initial point $x$ for the underlying process $X^\eps_x$, as $\eps \rightarrow 0$, then our claim follows.
The idea is to bound \eref{e:boundWanted} by the `worst-case scenario' obtained by replacing the process $Z^\eps$ by the process
$V^\eps$ described in \eref{e:defV}.

One technical problem that arises is that it is tricky to get pathwise control on the behaviour of $V$ due to the discontinuity
of its drift. We therefore
first compare $Z^{\eps}$ with the process $U_{x}^{\eps}$ solution to
\begin{equ}[e:defU]
dU^\eps = b^\eps_U(U^\eps)\,dt + dB_U^\eps(t)\;,\quad U_x^\eps(0) = y\;,
\end{equ}
where $B_U^\eps$ is a Brownian motion to be determined and $b^{\eps}_{U}$ is the Lipschitz continuous odd function defined on the positive real numbers by
\begin{equs}
b^{\eps}_{U}(x) =
\begin{cases}
-\frac{C_{V}}{\eps^{2}}\,x&\text{for $|x| \le \eps$,} \\
 -\frac{C_{V}}{\eps}&\text{for $\eps < x \leq (2+ \tilde \eta) \eps$,}\\
 -\frac{C_{V}}{\eps^{2}}\,((3+ \tilde \eta)\eps -x)&\text{for $(2+ \tilde \eta) \eps<x\leq(3+ \tilde \eta) \eps$,}\\
 0 &\text{otherwise.}
 \end{cases}
\end{equs}
The SDE \eref{e:defU} satisfies pathwise uniqueness, which is why we are using it as an intermediary between $Z^{\eps}$ and $V^{\eps}$. Now, what we are going to do is, given a realisation of the Brownian motion $B^\eps$ driving $Z^\eps$ in \eref{e:equZ}, to choose the Brownian motion $B_U^\eps$ driving
$U_{x}^{\eps}$ by changing the sign of the increments in such a way that
the absolute value of $U_{x}^{\eps}$ is always less than or equal
to $|Z_{x}^{\eps}| + 2\eps$. By pathwise uniqueness, we are indeed free to choose the Brownian motion in \eref{e:defU}. The choice of the Brownian motion is the content of the following lemma:

\begin{lem}\label{lem:coupleBM}
For every initial condition $x$, there exists a map $B^\eps \mapsto B_U^\eps$ that preserves Wiener measure and such that
 $|U_{x}^{\eps}| \leq |Z_{x}^{\eps}| + 2\eps$ almost surely.
 In particular, it follows that \eref{e:boundWanted} is bounded by
\begin{equ}
	\mathbb{E}_{x}\Bigl(\int_{0}^{\infty}e^{-\lambda t}1_{(-\delta',\,\delta')}\bigl(U_{x}^{\eps}(t)\bigr)\,dt\,\Bigr)\;, \label{8.2}
\end{equ}
where $\delta' = \delta+2\eps$.
\end{lem}
\begin{proof}
The construction works in the following way. Consider first the processes driven by the same realisation $B^\eps$ and define a
stopping time $\tau_0$ by $\tau_{0} = \inf\{t>0: |U^{\eps}_{x}(t)| = |Z^{\eps}_{x}(t)| + 2\eps\}$. This stopping time is strictly positive
and one has $|U^{\eps}_{x}(\tau_0)| \ge 2\eps$. For times after $\tau_0$, we determine $B_U^\eps$ by
\begin{equs}
B_U^\eps(t) = B_U^\eps(\tau_0) + \sign\bigl(U^{\eps}_{x}(\tau_0)\bigr) \int_{\tau_0}^t \sign \bigl(Z^{\eps}_{x}(s)\bigr)\,dB^\eps(s) \;,
\end{equs}
and we introduce the stopping time $\sigma_1 =  \inf\{t>\tau_{0}: |U^{\eps}_{x}| = \eps\}$. Since by construction
$U^\eps$ does not change sign between $\tau_0$ and $\sigma_1$, it then follows from
the It\^o-Tanaka formula that up to $\sigma_1$ one has
\begin{equs}
d |U^\eps| &= b_U^\eps\bigl(|U^\eps|\bigr)\,dt + \sign \bigl(Z^{\eps}\bigr)\,dB^\eps(t)\;,\\
d |Z^\eps| &= \sign \bigl(Z^{\eps}\bigr) \tilde b^\eps\bigl(Z^\eps\bigr)\,dt + \sign \bigl(Z^{\eps}\bigr)\,dB^\eps(t) + dL(t)\;,
\end{equs}
for some local time term $L$. Since the local time term always yields positive contributions and since it follows from
the definition that $b_U^\eps(u) \le \sign(z) \tilde b^\eps(z)$ for $|u| \ge \eps$ and $|u| \le |z| + 2\eps$, we can apply a simple comparison result for SDEs to conclude that the inequality $|U_{x}^{\eps}| \leq |Z_{x}^{\eps}| + 2\eps$ holds almost surely between times
$\tau_0$ and $\sigma_0$.

We then drive again both processes by the same noise and
define as before $\tau_1$ by  $\tau_{1} = \inf\{t>\sigma_1: |U^{\eps}_{x}(t)| = |Z^{\eps}_{x}(t)| + 2\eps\}$. Note that
$\tau_1 > \sigma_1$ almost surely since one has $|U^{\eps}_{x}(\sigma_0)| \le |Z^{\eps}_{x}(\sigma_0)| + \eps$.
We then apply the previous construction iteratively, so that, setting $\sigma_0 = 0$, we have constructed $B_U^\eps$ by
\begin{equs}
B_U^\eps(t) = \int_{0}^{t}  \Bigl(\sum_{n=0}^{\infty} 1_{[\sigma_{n},\tau_{n})}(s) +\sum_{n=0}^{\infty} 1_{[\tau_{n},\sigma_{n+1})}(s) \,\sign\bigl(U^{\eps}_{x}(\tau_{n})Y^{\eps}_{x}(s)\bigr)\Bigr) \,dB^\eps (s) \;.
\end{equs}
Since the process $U^\eps$ has finite quadratic variation and has to move by at least $\eps$ between
any two successive stopping times, our sequence of stopping times does converge to infinity, so that $B_U^\eps(t)$
is indeed a Brownian motion with the required property.
\end{proof}

In our next step, we compare the process $U_{x}^{\eps}$ that we just constructed
with the process $V_{x}^{\eps}$ defined in \eref{e:defV}, where we set $\hat \eta = 5 + \tilde \eta$.
Since the drift coefficient is bounded, it follows from an application of Girsanov's theorem like in
\cite[Corollary~IX.1.12]{MR1083357} that this SDE has a solution for some Brownian motion $B_{V}^\eps$, say.

We now fix $B_V^\eps$ and use it to construct a Brownian motion $B_U^\eps$ driving \eref{e:defU}
in such a way that the absolute value of $V_{x}^{\eps}$ always stays less than $|U_{x}^{\eps}| + 2\eps$:
 \begin{lem}
 There exists a map $B_V^\eps \mapsto B_U^\eps$ that preserves Wiener measure and such that
 $|V_{x}^{\eps}| \leq |U_{x}^{\eps}| + 2\eps$ for all times almost surely. In particular, (\ref{8.2}) is bounded by
\begin{equs}\mathbb{E}_{x}\Bigl(\int_{0}^{\infty}e^{-\lambda t}1_{(-\delta'',\,\delta'')}(V_{x}^{\eps}(t))\,dt\,\Bigr)\;, \label{8.1} \end{equs}
with $\delta'' = \delta'+2\eps$.
\end{lem}
 \begin{proof}
The argument is virtually identical to that of Lemma~\ref{lem:coupleBM}, so we do not reproduce it here.
\end{proof}

It now remains to show that:

\begin{lem}
The expression (\ref{8.1}) converges to $0$ uniformly in the initial point as $\eps \rightarrow 0$.
\end{lem}
\begin{proof}
We write $\eps$ for $5\eps+\eps\eta$ and $\delta$ for $\delta + 4\eps$ for ease of notation, but this has no bearing on the rates of convergence of the aforementioned quantities and hence on the calculation. We have the identity
    \begin{equs} \mathbb{E}_{x}\biggl[\int_{0}^{\infty}e^{-\lambda t}1_{(-\delta,\,\delta)}\bigl(V^{\eps}_{x}(t)\bigr)dt\biggr]
    &
    = \int_{0}^{\infty}e^{-\lambda t}P_{t}^\eps\bigl(1_{(-\delta,\,\delta)}\bigr)(x)\,dt\\
     &= (\lambda - \CL_V^{\eps})^{-1}1_{(-\delta,\delta)}(x) \label{e:resolvent}
\end{equs}
by the resolvent equation (see for example \cite[Chapter~1]{MR838085}), where $\CL_V^{\eps}$ is the generator of
the Markov semigroup $P_{t}^\eps$ associated to $V^{\eps}$.

We now proceed to computing this expression explicitly in order to show that its supremum tends to zero uniformly in $x$.
In order to keep notations simple, we assume for the remainder of this proof that $\eta = C_V = 1$, which can always be achieved by
rescaling space and redefining $\eps$. In this case, the solution $f(x) = (\lambda - \CL_V^{\eps})^{-1}1_{(-\delta,\delta)}(x)$
to the resolvent equation can be assembled piecewise on the intervals $(-\infty,-\delta)$, $(-\delta,-\eps)$, $(-\eps,0)$, $(0,\eps)$, $(\eps,\delta)$ and $(\delta,\infty)$ by making sure that it is $\CC^1$ at each junction. Owing to the symmetry of the problem, the
function $f$ will be an even function of $x$, hence we only have to analyze it on one side of the origin.

The general solution on each interval can be written as
\begin{equs}
f(x) =
\left\{\begin{array}{cl}
	B_0 e^{-\sqrt{2\lambda} x} & \text{for $x \ge \delta$,} \\
	{1\over \lambda} + A_1 e^{\sqrt{2\lambda} x} + B_1 e^{-\sqrt{2\lambda} x} & \text{for $\eps \le x \le \delta$,} \\
	{1\over \lambda} + \eps^2 A_2 e^{\gamma_1 x} + B_2 e^{-\gamma_2 x} & \text{for $x \le \eps$,}
\end{array}\right.
\end{equs}
where
\begin{equs}
\gamma_{1} &= \biggl(\frac{1}{\varepsilon^{2}}+2\lambda\biggr)^{\frac{1}{2}} + \frac{1}{\varepsilon} =  {2\over \eps} + \CO(\eps)\\
    \gamma_{2} & = \biggl(\frac{1}{\varepsilon^{2}}+2\lambda\biggr)^{\frac{1}{2}} - \frac{1}{\varepsilon} = \lambda \eps + \CO(\eps^2)\;.
\end{equs}
The reason for the somewhat strange choice of adding an explicit factor $\eps^2$ in front of $A_2$ is justified
\textit{a posteriori} by noting that with this scaling, the matching conditions at $\eps$ and $\delta$
(as well as the fact that the derivative should vanish
at the origin) yield the following linear system:
\begin{equ}
M \begin{pmatrix}
B_0 \\ A_1 \\ B_1 \\ A_2 \\ B_2
\end{pmatrix} =
\begin{pmatrix}
0 \\ 0 \\ 0 \\ -{1\over \lambda} \\ 0
\end{pmatrix}\;,\qquad
M = \begin{pmatrix}
0 & 0 & 0 &  2 & -\lambda \\
0 & -1 & -1 & 0 & 1 \\
0 & -1 & 1 & 0 & 0 \\
-1 & 1 & 1 & 0 & 0 \\
1 & 1 & -1 & 0 & 0
\end{pmatrix} + \CO(\delta)\;.
\end{equ}
To lowest order in $\eps$ and $\delta$, this can easily be solved exactly, yielding
\begin{equ}
(B_0,A_1,B_1,A_2,B_2) = -{1\over 2\lambda}\bigl(0 , 1, 1, \lambda , 2\bigr) +  \CO(\delta)\;.
\end{equ}
Inserting this into the expression for $f$ shows that $\sup_{x\in \R} |f(x)| = \CO(\delta) = \CO(\sqrt \eps)$, thus completing the proof.
\end{proof}
With the lemma regarding the resolvent calculation above, the proof of Proposition~\ref{prop:timeint} is complete.
\end{proof}

We finally show that

\begin{prop}\label{prop:exit}
For every $c>0$, the exit probabilities from the interval $(-\delta,\delta)$ satisfy the bound
\begin{equs}
\mathbb{P}_{x} \bigl[Y^{\eps}(\sigma^{\delta}) \in I_{i}\bigr] = p_{i} + \CO(\sqrt\eps)\;,
\end{equs}
uniformly for $x \in [-c\eps,c\eps]$.
\end{prop}
 \begin{proof}
For the proof of this result, it turns out to be simpler to consider the original
process $X^{\eps}(t)$.

Whenever $Y^{\eps}(t)$ exits
the set $(-\delta,\delta)$, due to the deterministic
relationship between the processes, $X^{\eps}(t)$ exits
a set $(-\delta',\delta'')$, where $\delta'$ and $\delta''$
 are contained in the interval
$(\delta - \eps\|g\|_{\infty},\delta + \eps\|g\|_{\infty})$.
Therefore, we just look at the exit of $X^{\eps}(t)$ from
an interval of this form as the computations are much easier to carry out. This is due to the simpler form of the scale function for $X^{\eps}(t)$ compared with that of $Y^{\eps}(t)$. It follows
from \cite[Exercise~VII.3.20]{MR1083357}  that the scale function of the
diffusion on $\mathbb{R}$ with generator $\CL\, =\,
\frac{1}{2}\,\sigma^{2}(x)\frac{d^{2}}{dx^{2}}\, +\,
b(x)\frac{d}{dx}$ is given by,
    \begin{equs} s(x) =
    \int_{c}^{x}\exp{\Bigl(-\int_{c}^{y}2b(z)\sigma^{-2}(z)dz\Bigr)}\,dy \end{equs}
    where $c$ is an arbitrary point in $\mathbb{R}$. Recall that the scale function
    of a real-valued process is a continuous, strictly increasing function such that for any $a < x < b$ in the set where the Markov process takes its values, one has
\begin{equs} \mathbb{P}_{x}\bigl(T_{b}<T_{a}\bigr) =
    \frac{s(x)-s(a)}{s(b)-s(a)} \;,
\end{equs}
where $T_{a}$, $T_{b}$ are the first hitting times of the points $a$ and $b$ respectively.
For $X^{\eps}(t)$ we have that $\sigma = 1$ and the drift
is equal to
$\frac{1}{\eps}\,b(x/\eps)$. We are from now on going to use the notation
$q_\eps(y) = \bigl(1 + \eps g(\cdot/ \eps)\bigr)^{-1}(y)$ for the transformation that allows recovery of $X^\eps$ from $Y^\eps$. We also denote by $T_a$ the first hitting time of the point $a \in \R$ by the process $X^\eps$. We also use the shorthand notation
\begin{equ}
F_\eps(u) = \exp \Bigl({-2\int_{0}^{u}\frac{1}{\eps}\,b(z/\eps)\,dz}\Bigr)\;,
\end{equ}
so that the scale function for $X^\eps$ is given by $s(z) = \int_0^z F_\eps(y)\,dy$.

With this notation at hand, we have, for $x \in (-\delta, \delta)$, that,
denoting the escape time of $Y^{\eps}(t)$ from
$(-\delta, \delta)$ by $\sigma^{\delta}$,
\begin{equs}
\mathbb{P}_{x}\bigl(Y^{\eps}(\sigma^{\delta}) \in I_+\bigr)
     &= \mathbb{P}_{x}\bigl(T_{\delta''}<T_{-\delta'}\bigr)
     = \frac{s\bigl(q_{\eps}(x)\bigr) - s(-\delta')}{s(\delta'') - s(-\delta')}\\
     &= \frac{\int_{0}^{q_{\eps}(x)}F_\eps(y)\,dy - \int_{0}^{-\delta'}F_\eps(y)\,dy} { \int_{0}^{-\delta''}F_\eps(y)\,dy -  \int_{0}^{-\delta'}F_\eps(y)\,dy}
     = \frac{ \int_{-\delta'}^{q_\eps(x)}F_\eps(y)\,dy} { \int_{-\delta'}^{\delta''}F_\eps(y)\,dy}\;.
\end{equs}
Noting that $F_\eps$ has the scaling property $F_\eps(u) = F_1(u/\eps)$, we thus obtain the identity
\begin{equ}[e:boundHP]
\mathbb{P}_{x}\bigl(Y^{\eps}(\sigma^{\delta}) \in I_+\bigr)
     = \frac{ \int_{-\delta'}^{q_\eps(x)}F_1(y/\eps)\,dy} { \int_{-\delta'}^{\delta''}F_1(y/\eps)\,dy}
     = \frac{ \int_{-\delta/\eps}^{0}F_1(y)\,dy + \CO(1)} { \int_{-\delta/\eps}^{\delta/\eps}F_1(y)\,dy + \CO(1)}\;,
\end{equ}
where we used the fact that $q_\eps(x) = \CO(\eps)$ and $F_1$ is uniformly bounded, due to the fact that the functions
$b_\pm$ are centred by assumption.

Note now that the effective diffusion coefficients $C_\pm$ can alternatively be expressed as \cite[Sec~13.6]{MR2382139}
\begin{equs}
C_+^2 = \biggr[\int_{\eta}^{\eta+1}\exp \bigl(-2V(u)\bigr)\,du\;\int_{\eta}^{\eta+1}\exp\bigl(2V(u)\bigr)\,du\;\biggr]^{-1}\;,
\end{equs}
and similarly for $C_-$.
Therefore, since $F_1$ is periodic away from $[-\eta, \eta]$, it follows immediately from the definitions of $\lambda_\pm$ and $C_{\pm}$ that
\begin{equ}
\int_{-N}^0 F_1(y)\,dy = \frac{1}{C_{-}^{2} \lambda_-} N + \CO(1)\;,\qquad
\int_{0}^N F_1(y)\,dy = \frac{1}{C_{+}^{2}  \lambda_+} N + \CO(1)\;.
\end{equ}
Since $\delta = \sqrt \eps$, combining these bounds with \eref{e:boundHP}
implies that $\mathbb{P}_{x}\bigl(Y^{\eps}(\sigma^{\delta}) \in I_+\bigr)
= C_{+}^{2}\lambda_+ / (C_{-}^{2}\lambda_+ + C_{+}^{2}\lambda_-) + \CO(\sqrt \eps)$, from which the requested bound follows.
\end{proof}

Combining Propositions~\ref{prop:firstbound}, \ref{prop:timeint}, and \ref{prop:exit} completes the proof of
Proposition~\ref{prop:assum}.
\end{proof}

 \section{Uniqueness and characterisation of the martingale problem}
\label{sec:unique}

To conclude this article, we show that solution to the martingale problem corresponding to $A$ is unique and is indeed given by the variant of skew
Brownian motion constructed in the introduction.

   The skew Brownian motion $B_p$ of `skewness' parameter $p$ is known to have generator $\CL_p = {1\over 2}\frac{d^{2}}{dx^{2}}$ on the set of functions that are continuous, twice continuously differentiable except at the origin where we have $p f'(0^{+}) = (1-p)f'(0^{-})$.
(See for example the review article \cite{MR2280299}.)
The process $B_{C_\pm,p}$ constructed in the introduction is given by
\begin{equ}
B_{C_\pm,p}(t) = G(B_p(t))\;,\qquad G(x) =
\left\{\begin{array}{cl}
	C_+ x & \text{if $x \ge 0$,} \\
	C_- x & \text{otherwise.}
\end{array}\right.
\end{equ}
Since $G$ is a continuous bijection, this is again a strong Markov process and it has generator given by
$Af =  \bigl(\CL_p (f \circ G)\bigr)\circ G^{-1}$ with $\cD(A) = \{f \,:\, f\circ G \in \cD(\CL_p)\}$. Using the relation \eref{e:relplambda},
it is now a straightforward
calculation to show that $\cD(A)$ consists precisely of those functions satisfying the derivative condition
in \eref{e:dcon}, hence we will have
\begin{equs}
\frac{p_{+}}{p_{-}} = \frac{C_{+}p}{C_{-}(1-p)}\;.
\end{equs}
This provides a clue as to how to show uniqueness easily. Let $Y$ denote any solution of the martingale problem corresponding to $A$ acting on $\cD(A)$, $Y$ also represents any possible limit point of the family $Y^{\eps}$ as $\eps\rightarrow0$. Defining $g$ like $G$, but with constants ${p/ C_+}$ and ${(1-p) / C_-}$ instead of $C_+$ and $C_-$, 
we note that $g$ satisfies the derivative condition at $0$ imposed
for elements of $\cD(A)$. Since $g$ doesn't vanish at infinity, we have $g \notin \cD(A)$, but we can approximate $g$ by a sequence $g_{n}\in\cD(A)$
such that $g=g_{n}$ on $[-n,n]$. Indeed, given $t>0$, the probability of escaping from $[-n,n]$ before time $t$ tends to zero as $n\rightarrow\infty$
uniformly for $\eps \in (0,1]$ by tightness. As a consequence, the process $V=g(Y)$ satisfies the SDE
\begin{equ}
V(t) = \int_{0}^{t} A_+ 1_{\{V(s)>0\}}+ A_- 1_{\{V(s)<0\}}\,dB_{s}\;,
\end{equ}
where $A_\pm$ are some constants.
It is known \cite[Theorem 2.1]{MR2203887} that this equation has a pathwise unique solution which in particular implies uniqueness in law.
Since $g$ is an invertible map, this immediately implies that $Y$ is unique in law and one can check that $Y$ is indeed the variant of skew Brownian motion
constructed in the introduction.

\bibliographystyle{Martin}
\bibliography{./onedimref}

\def\cprime{$'$}
\begin{thebibliography}{BMP05}
\expandafter\ifx\csname url\endcsname\relax
  \def\url#1{\texttt{#1}}\fi
\expandafter\ifx\csname urlprefix\endcsname\relax\def\urlprefix{URL }\fi

\bibitem[ACP03]{MR2021045}
\textsc{G.~Allaire}, \textsc{Y.~Capdeboscq}, and \textsc{A.~Piatnitski}.
\newblock Homogenization and localization with an interface.
\newblock \emph{Indiana Univ. Math. J.} \textbf{52}, no.~6, (2003), 1413--1446.

\bibitem[BC05]{MR2203887}
\textsc{R.~F. Bass} and \textsc{Z.-Q. Chen}.
\newblock One-dimensional stochastic differential equations with singular and
  degenerate coefficients.
\newblock \emph{Sankhy\=a} \textbf{67}, no.~1, (2005), 19--45.

\bibitem[BLP78]{lions}
\textsc{A.~Bensoussan}, \textsc{J.~Lions}, and \textsc{G.~Papanicolaou}.
\newblock \emph{Asymptotic analysis of periodic structures}.
\newblock North-Holland, Amsterdam, 1978.

\bibitem[BMP05]{MR2126985}
\textsc{A.~Bench{\'e}rif-Madani} and \textsc{{\'E}.~Pardoux}.
\newblock Homogenization of a diffusion with locally periodic coefficients.
\newblock In \emph{S\'eminaire de {P}robabilit\'es {XXXVIII}}, vol. 1857 of
  \emph{Lecture Notes in Math.},  363--392. Springer, Berlin, 2005.

\bibitem[EK86]{MR838085}
\textsc{S.~N. Ethier} and \textsc{T.~G. Kurtz}.
\newblock \emph{Markov processes}.
\newblock Wiley Series in Probability and Mathematical Statistics: Probability
  and Mathematical Statistics. John Wiley \& Sons Inc., New York, 1986.
\newblock Characterization and convergence.

\bibitem[FW93]{MR1245308}
\textsc{M.~I. Freidlin} and \textsc{A.~D. Wentzell}.
\newblock Diffusion processes on graphs and the averaging principle.
\newblock \emph{Ann. Probab.} \textbf{21}, no.~4, (1993), 2215--2245.

\bibitem[HK98]{MR1632945}
\textsc{B.~M. Hambly} and \textsc{T.~Kumagai}.
\newblock Heat kernel estimates and homogenization for asymptotically
  lower-dimensional processes on some nested fractals.
\newblock \emph{Potential Anal.} \textbf{8}, no.~4, (1998), 359--397.

\bibitem[HM10]{MultiD}
\textsc{M.~Hairer} and \textsc{C.~Manson}.
\newblock Periodic homogenization with an interface: the multi-dimensional
  case, 2010.
\newblock Preprint.

\bibitem[KK96]{MR1376343}
\textsc{T.~Kumagai} and \textsc{S.~Kusuoka}.
\newblock Homogenization on nested fractals.
\newblock \emph{Probab. Theory Related Fields} \textbf{104}, no.~3, (1996),
  375--398.

\bibitem[Koz93]{MR1218305}
\textsc{S.~M. Kozlov}.
\newblock Harmonization and homogenization on fractals.
\newblock \emph{Comm. Math. Phys.} \textbf{153}, no.~2, (1993), 339--357.

\bibitem[Kum04]{MR2083710}
\textsc{T.~Kumagai}.
\newblock Homogenization on finitely ramified fractals.
\newblock In \emph{Stochastic analysis and related topics in {K}yoto}, vol.~41
  of \emph{Adv. Stud. Pure Math.},  189--207. Math. Soc. Japan, Tokyo, 2004.

\bibitem[Lej06]{MR2280299}
\textsc{A.~Lejay}.
\newblock On the constructions of the skew {B}rownian motion.
\newblock \emph{Probab. Surv.} \textbf{3}, (2006), 413--466 (electronic).

\bibitem[Owh03]{MR2016606}
\textsc{H.~Owhadi}.
\newblock Anomalous slow diffusion from perpetual homogenization.
\newblock \emph{Ann. Probab.} \textbf{31}, no.~4, (2003), 1935--1969.

\bibitem[PS08]{MR2382139}
\textsc{G.~A. Pavliotis} and \textsc{A.~M. Stuart}.
\newblock \emph{Multiscale methods}, vol.~53 of \emph{Texts in Applied
  Mathematics}.
\newblock Springer, New York, 2008.
\newblock Averaging and homogenization.

\bibitem[PV82]{MR659505}
\textsc{G.~C. Papanicolaou} and \textsc{S.~R.~S. Varadhan}.
\newblock Diffusions with random coefficients.
\newblock In \emph{Statistics and probability: essays in honor of {C}. {R}.
  {R}ao},  547--552. North-Holland, Amsterdam, 1982.

\bibitem[RY91]{MR1083357}
\textsc{D.~Revuz} and \textsc{M.~Yor}.
\newblock \emph{Continuous martingales and {B}rownian motion}, vol. 293 of
  \emph{Grundlehren der Mathematischen Wissenschaften [Fundamental Principles
  of Mathematical Sciences]}.
\newblock Springer-Verlag, Berlin, 1991.

\bibitem[SV79]{MR532498}
\textsc{D.~W. Stroock} and \textsc{S.~R.~S. Varadhan}.
\newblock \emph{Multidimensional diffusion processes}, vol. 233 of
  \emph{Grundlehren der Mathematischen Wissenschaften [Fundamental Principles
  of Mathematical Sciences]}.
\newblock Springer-Verlag, Berlin, 1979.

\bibitem[Zhi95]{MR1474004}
\textsc{V.~V. Zhikov}.
\newblock Connectedness and homogenization. {E}xamples of the fractal
  conductivity.
\newblock In \emph{Homogenization and applications to material sciences
  ({N}ice, 1995)}, vol.~9 of \emph{GAKUTO Internat. Ser. Math. Sci. Appl.},
  421--430. Gakk\=otosho, Tokyo, 1995.

\end{thebibliography}
\end{document}